\documentclass[11pt, a4paper]{article}

\usepackage{amsmath, amsthm, amsfonts, amssymb}

\usepackage{setspace}
\usepackage{fullpage}
\usepackage{enumitem}
\usepackage{graphicx}

\usepackage{floatpag}
\usepackage[dvipsnames]{xcolor}

\usepackage[initials]{amsrefs} 

\usepackage{float}
\usepackage{caption}
\usepackage{subcaption} 

\usepackage{hyperref}

\hypersetup{colorlinks=true,
	citecolor=blue,
	filecolor=blue,
	linkcolor=blue,
	urlcolor=blue
}
\usepackage{cleveref}

\usepackage{comment}

\newtheorem{theorem}{Theorem} 
\newtheorem{lemma}[theorem]{Lemma}
\newtheorem{corollary}[theorem]{Corollary}

\newtheorem{proposition}[theorem]{Proposition}

\newcommand{\cA}{\mathcal{A}}
\newcommand{\cB}{\mathcal{B}}

\newcommand{\cH}{\mathcal{H}}
\newcommand{\cX}{\mathcal{X}}

\newcommand{\ex}{{\rm  ex}}
\newcommand{\pre}{{\rm pre}}
\newcommand{\suf}{{\rm suf}}

\newcommand{\st}{$\mathord{\star}$}

\begin{document}
	\title{On graphs embeddable in a layer of a hypercube and their extremal numbers}
	\date{\vspace{-5ex}}
	\author{
		Maria Axenovich
		\thanks{Karlsruhe Institute of Technology, Karlsruhe, Germany;
			email:
			\mbox{\texttt{maria.aksenovich@kit.edu}}
}
 ~~ Ryan R.  Martin  
\thanks{Iowa State University, Ames, Iowa, U.S.A.;
			email:
			 \texttt{rymartin@iastate.edu}}						
~~Christian Winter			
\thanks{Karlsruhe Institute of Technology, Karlsruhe, Germany;
			email:
			\mbox{\texttt{christian.winter@kit.edu}}
}}
			
	\maketitle
	
	\begin{abstract}
	A graph is cubical if it is a subgraph of a hypercube. For a cubical graph $H$ and a hypercube $Q_n$, $\ex(Q_n, H)$ is the largest number of edges in an $H$-free subgraph of $Q_n$.
If $\ex(Q_n, H)$ is at least  a positive proportion of the number of edges in $Q_n$, then $H$ is said to have positive Tur\'an density in the hypercube; otherwise it has zero Tur\'an density.   Determining $\ex(Q_n, H)$ and even identifying whether $H$ has positive or zero Tur\'an density remains a widely open question for general $H$. 
	
 In this paper we focus on layered graphs, i.e., graphs that  are contained in an edge-layer of some hypercube. Graphs $H$ that are not layered have positive Tur\'an density because one can form an $H$-free subgraph of $Q_n$ consisting of edges of every other layer.  For example, a $4$-cycle is not layered and has positive Tur\'an density.  
 
However, in general it is not obvious what properties layered graphs have. We give a characterisation of layered graphs in terms of edge-colorings.
We show that most non-trivial subdivisions have zero Tur\'an density, extending known results on zero Tur\'an density of  even cycles of length at least $12$ and of length $8$. However,  we prove that there are cubical graphs of girth $8$ that are not layered and thus having positive Tur\'an density. 
The cycle of length $10$ remains the only cycle for which it is not known whether its Tur\'an density is positive or not.   We prove that $\ex(Q_n, C_{10})= \Omega(n2^n/ \log^a n)$, for a constant $a$, showing that the extremal number for a $10$-cycle behaves differently from any other cycle of zero Tur\'an density. 
	\end{abstract}

	\section{Introduction}
	The {\it hypercube} $Q_n$, where $n$ is a natural number, is a graph on  a vertex set $\{A: A\subseteq [n]\}$ and an edge set consisting of all pairs $\{A,B\}$, where $A\subseteq B$ and $|A|=|B|-1$. Here, $[n]= \{1, \ldots, n\}$. We often identify vertices of $Q_n$ with binary vectors that are indicator vectors of respective sets.  If a graph is a subgraph of $Q_n$, for some $n$, it is called {\it cubical}. We denote the number of vertices and the number of edges in a graph $H$ by  $|H|$ and $||H||$, respectively. \\
	
For a graph $H$, let the {\it extremal number}  of $H$ in $Q_n$, denoted 	$\ex(Q_n, H)$,  be the largest number of edges in a subgraph $G$ of  $Q_n$ such that there is no subgraph of $G$ isomorphic to $H$. 
A graph $H$ is said to have {\it zero Tur\'an density in the hypercube} if $\ex(Q_n, H) = o(||Q_n||)$. Otherwise, we say that $H$ has  {\it positive Tur\'an density in the hypercube}. Note that by using a standard double counting argument,  the sequence $\ex(Q_n, H)/||Q_n||$ is non-increasing, thus the above density notions are well-defined. When clear from context, we simply  say {\it Tur\'an density}  instead of Tur\'an density in a hypercube. 
The behaviour of the function $\ex(Q_n, H)$ is not well understood in general and it is not even known what graphs have positive or zero Tur\'an density.
Currently, the only known cubical graphs of positive Tur\'an density are those containing a $4$- or a $6$-cycle as a subgraph.
Conlon \cite{C} observed a connection between extremal numbers in the hypercube and classical extremal numbers for uniform hypergraphs. That permitted the determination of a large class of  graphs with zero Tur\'an density. For more results on extremal numbers in the hypercube, see  \cite{baber, balogh, AKS, ARSV, AM,TW, O}.\\

Another class of graphs that are of particular importance as a superset of all graphs of zero Tur\'an density corresponds to  so-called layered graphs. The $k$th {\it vertex layer}, denoted $V_k$,  of $Q_n$ is $\binom{[n]}{k}$, the set of all vertices that are $k$-element subsets of $[n]$. The $k$th {\it edge layer} of $Q_n$ is the subgraph of $Q_n$ induced by the $k$th and $(k-1)$st vertex layers.  For other standard graph theoretic notions, we refer the reader to Diestel \cite{D}. 
 A cubical graph is called {\it layered} if it is a subgraph of some edge layer of $Q_n$, for some $n$.  Note for example, that $C_4$ is not layered and $C_{2\ell}$  is layered for  any $\ell\geq 3$. It is an easy observation that cubical graphs that are not layered have positive Tur\'an density. Indeed,  a subgraph of $Q_n$ that is a union of its even (or odd) edge layers contains only layered connected graphs as subgraphs.   \\

In this paper, we focus on layered graphs.  First, we  give a characterization of layered graphs in terms of edge-colorings. 
We say that an edge-coloring of a graph is {\it nice} if for any cycle, each color appears an even number of times and in any path with at least one edge there is a color that appears odd number of times. We say that an edge coloring of a graph $G$ is {\it very nice} if  it is nice and for any two edges of the same color, any path between them that has no edges of that color has an even length. We extend a result by Havel and Moravek \cite{HM} to layered graphs. 

\begin{theorem}\label{very-nice-intro}
A graph is layered if and only if it has a very nice edge-coloring. 
\end{theorem}

Theorem \ref{very-nice-intro}  shows in particular that graphs with no very nice coloring 
have positive Tur\'an density.   A natural question to consider is whether there are sparse cubical graphs that have positive Tur\'an density. We show that most subdivisions have zero Tur\'an density, but there are graphs of girth at least eight that have positive Tur\'an density.  Let $K_t$ and $K_{t,t}$ be complete and balanced complete bipartite graphs on $t$ and $2t$ vertices, respectively. For a graph $G$ and a positive integer $k$, let $T_k(G)$ be a $k$-subdivision of $G$, i.e., a graph obtained from $G$ by subdividing each edge with $k$ new vertices. Since an even subdivision of an odd cycle is an odd cycle, that is not cubical, we consider even subdivisions of bipartite graphs only. 

\begin{theorem}\label{subdivisions}
Let $k$ and $t$ be positive integers.
Then $T_{2k+1}(K_t)$ and $T_{2k}(K_{t,t})$ are  layered. Moreover, 
$\ex(Q_n, T_{2k+1}(K_t)) = O(n^b2^n)= o(||Q_n||)$, where $b=1-(k+1)^{-1}t^{-k}$.
If, in addition,  $k\geq 4$ is even,  then 
$\ex(Q_n, T_{2k}(K_{t,t})) = O( n^{b'}2^n) = o(||Q_n||)$, where $b'= 1 -(2t^3+4t)^{-1}k^{-t^2}$.
\end{theorem}

\begin{theorem}\label{girth8}
There is a cubical graph of girth $8$ that is not layered. 
\end{theorem}

 A lot of research was done on  even cycles and their extremal numbers in a hypercube. Here, a $2\ell$-cycle is  denoted $C_{2\ell}$.
 The fact that $\ex(Q_n, C_4) = \Omega(||Q_n||)$ and $\ex(Q_n, C_6) = \Omega(||Q_n||)$  was shown by Chung~\cite{chung}, Conder~\cite{conder}, and Brass {et al.}~\cite{brass}. Chung \cite{chung} showed that $\ex(Q_n, C_{4k})=o(||Q_n||)$, for any integer $k\geq 2$.
F\"uredi and \"Ozkahya~\cite{FO, FO2}   extended Chung's results by showing that  $\ex(Q_n, C_{4k+2})=o(||Q_n||)$, for any integer $k\geq 3$.
Thus $C_{2\ell}$ has zero Tur\'an density for $\ell= 4$ and $\ell\geq 6$. 
Considering more specific upper bounds for cycles with zero Tur\'an density,  Conlon \cite{C}  proved for $k \geq 2$ that
$\ex(Q_n, C_{4k}) \leq c_k  n^{-1/2 + 1/(2k)}  ||Q_n||$.   Improving on results of F\"uredi and \"Ozkahya~\cite{FO, FO2}, Axenovich \cite{MA}  showed that  for an odd integer $\ell \geq 7$, $\ex(Q_n, C_{2\ell})= O\left(n^{5/6 + 1/(3(\ell-3))} 2^n\right)$. Tomon \cite{T} independently proved a better upper bound for large $\ell$:
$\ex(Q_n, C_{2\ell})= O(n^{2/3+ \delta} 2^n), $ for some $\delta = O((\log \ell) /\ell)$.\\  
 
 It remains unknown whether $C_{10}$ has zero or positive Tur\'an density.
 While we still could not answer this question we improve on the known lower bounds of $\ex(Q_n, C_{10})$: 

\begin{theorem} \label{thm:C10}
$\ex(Q_n, C_{10}) = \Omega\left( \frac{n}{\log^a n} 2^n\right)$, where $a = \log_2 3$.
\end{theorem}

The rest of the paper is structured as follows. We prove Theorem \ref{very-nice-intro} as an immediate corollary of Theorem \ref{very-nice} in Section \ref{colorings}.
In Section \ref{sec:subdivisions} we address subdivisions and prove Theorems \ref{complete} and \ref{complete-bipartite}, that imply Theorem \ref{subdivisions}.
Theorem \ref{girth8} is proved   in Section \ref{girth} and Theorem \ref{thm:C10} is proved   in Section \ref{sec:C10}. 
We give some density properties of layered graphs in Section \ref{density}.
Section \ref{conclusions} contains concluding remarks and open questions. In Appendix A we present an alternative proof for an upper bound on extremal numbers for graphs of zero Tur\'an density. In Appendix B we provide a symmetric layered embedding of a hypercube.\\

 After this paper was accepted for publication, two of the questions from this paper were answered. First, it was shown by Grebennikov and Marciano \cite{GM} that $C_{10}$ has positive Tur\'an density in the hypercube using a construction for daisy-free hypergraphs by  Ellis, Ivan, and Leader \cite{EIL}.  Second,   Behague, Leader,  Morrison,  and Williams \cite{BLMW}   showed that there is a cubical graph of arbitrarily high girth that is not layered.

\vskip 1cm

\section{Characterisation of layered graphs in terms of very nice colorings, proof of Theorem \ref{very-nice-intro}}\label{colorings}

	Recall that  an edge-coloring of a graph is {\it nice} if, for any cycle, each color appears an even number of times and in any path with at least one edge there is a color that appears an  odd number of times. An edge-coloring of a graph $G$ is {\it very nice} if  it is nice and, for any two edges of the same color, any path between them that has no edges of that color has an even length. 
	
	\begin{theorem}[Havel and Moravek \cite{HM}]
	A graph is cubical if and only if there is a nice edge-coloring of the graph. 
	\end{theorem}

	Here, we extend this characterisation to layered graphs. Recall that the distance between two edges in a connected graph is the length of a shortest path between some endpoint of one edge and some endpoint of the other edge. Similarly, the distance between a vertex $v$  and a set of edges $S$ is the smallest distance between $v$ and an edge from $S$. For an edge of $Q_n$, let its {\it direction} be the coordinate at which its endpoints differ.  We shall also represent an edge $AB$, $A\subseteq B$ in $Q_n$ by a sequence of length $n$, 
where the $i$th position is occupied by $0$ if $i\notin B$,  by $1$ if $i\in A$ and by $\star$ if $i\in B\setminus A$.  We call this a \textit{star representation} and refer to a position occupied by a $\star$ as a {\it star position}, that in turn corresponds to the direction of the edge.  A {\it color class} in an edge-coloring of a graph is a set of all edges having the same color. The following theorem immediately implies Theorem \ref{very-nice-intro}. The following theorem contains some additional properties of very nice colorings that are of independent interest.

\begin{theorem}\label{very-nice}
A graph is layered if and only if it has a very nice edge-coloring. Moreover, if a graph is embedded in a layer and its edges are assigned colors corresponding to directions of the edges, then this coloring is very nice. In addition, any color class in a very nice coloring of a connected graph is a cut.
\end{theorem}

	\begin{proof}
	One direction of the proof is easy. Consider a connected graph $G$ with all edges in one layer of $Q_n$. 
	Let $c: E(G)\rightarrow [n]$ be a coloring such that $c(e)$ is equal to the direction of $e$.
	Then it is easy to see and was verified in  \cite{HM}, that $c$ is nice. 
	Consider  two edges $e$ and $e'$ of the same color and  a path between them not using that color. It is clear that the path must be of even length. \\
	
	For the other direction, consider a graph $G$ with a very nice coloring $c$. We can assume that it is connected. 
	Fix a vertex $v$ of $G$. Consider all color classes with even distance to $v$ and let $C^+$ be the set of colors on these color classes.  Let $C^-$ be the set of all other colors used on $G$. 
	We shall consider an embedding $f$  of $G$ that puts an edge in a direction corresponding to its color.  Assume that  $[n]=C^+ \cup C^-$. 
	Formally, let $f: V(G) \rightarrow V(Q_n)$ be defined as follows.
	Let $v$ be mapped to a vertex $f(v)= C^-$ in the $k$th layer, $V_k$, where $k= |C^-|$. Assume that a vertex $u$ has been mapped and $u'$ is a neighbor of $u$. We define $f(u')$ 
	to be the vertex in $Q_n$ such that $f(u)$ and $f(u')$ are adjacent and  the direction of $f(u)f(u')$ is equal to $c(uu')$. I.e., either $f(u)\setminus f(u')= c(uu')$ or $f(u') \setminus f(u)= c(uu')$.
  Let $G'$ be a graph resulted from this map, i.e., $V(G') = \{f(u): u\in V(G)\}$, $E(G') = \{f(u)f(u'): uu' \in E(G)\}$.\\

	 First of all, we have that the function $f$  is indeed an injective map into $Q_n$ preserving adjacencies exactly as  shown in  \cite{HM}. For completeness we repeat the argument here. The function $f$  is well-defined since for any $v,u$-path in $G$ and any color, the number of edges of that color has the same parity  among all such paths, since the coloring is nice.  Indeed, otherwise in the union of two paths with different parity of the number of edges of say color $j$, we would find a cycle with an odd number of edges colored $j$.   If $f(u)=f(u')$ for distinct vertices $u$ and $u'$,  consider a closed walk formed by taking a union of $f(v),f(u)$- and $f(v),f(u')$-paths in $G'$. A smallest cycle $C'$ in this walk containing $f(u)$ corresponds to a $u,u'$-path $P'$  in $G$. Let $W$ be the multiset of colors used by $c$ on $P'$.  By definition, $W$ corresponds to the multiset of directions of the edges of $C'$, so each direction in $W$ appears an even number of times.  However, the niceness of $c$ implies that  some color appears an odd number of times in $W$, a contradiction.
	So, the map is well-defined, injective, and it  clearly preserves adjacencies.\\

	Now, we shall show that $f$ maps the vertex set of $G$ into a subset of  $V_k \cup V_{k+1}$.  
	Consider an arbitrary vertex $u$ and a $v,u$-path $P$.  We claim by induction on the length of $P$ that $V(P)$ is mapped to a subset of $ V_k \cup V_{k+1}$. The basis for induction is trivial since $f(v)\in V_k$.  Let $P$ have length at least one, let $u'$ be the neighbor of $u$ in $P$,  and $P'=P-u$. Then by induction $V(P')$ is mapped onto a subset of $V_k\cup V_{k+1}$.  Let $j=c(uu')$.\\

Let $x_0$ be the number of edges between $v$ and the first edge of color $j$ on $P$, and let $x_1$ be the number of edges of color $j$ in $P'$.
Recall that  the number of edges between consecutive edges of color $j$ on any path is even. Thus we have that $||P'||$ is even if and only if $x_0 + x_1$ is even.
Furthermore, recall that $x_0$ is even if and only if $j\in C^+$, or equivalently $j\not\in f(v)$.
This implies that $j\not\in f(u')$ if and only if $x_0+x_1$ is even.
Therefore, if $x_0 + x_1$ is even, we have that $||P'||$ is even, thus $f(u')\in V_k$ and additionally $j\not\in f(u')$. Then by the rules of embedding $f(u) = f(u') \cup \{j\} \in V_{k+1}$. 
If $x_0+x_1$ is odd, then $||P'||$ is odd, so $f(u')\in V_{k+1}$, and $j\in f(u')$. Then by the rules of embedding $f(u) = f(u') -\{j\} \in V_{k}$.\\

To see that any color class in a very nice coloring of a connected graph $G$ is a cut, assume the opposite, i.e., assume that removing the edges of some color, $i$,  results in a connected graph $G'$. Then, the endpoints of some edge $e$ of color $i$ are connected by a path in $G'$. This path, together with the edge $e$ is a cycle with color $i$ represented on exactly one edge, thus contradicting the fact that the coloring is nice.
	\end{proof}

\section{Subdivisions - layered embeddings and extremal numbers, proof of Theorem \ref{subdivisions}}\label{sec:subdivisions}

We shall need some preliminary definitions and known results to prove Theorem \ref{subdivisions}.
\subsection{Partite representations, extremal numbers for hypergraphs,  and extremal numbers in a hypercube}

	We say that a subgraph $H$ of $Q_n$ has  a {\it $k$-partite representation}  $\cH$  if $H$ is isomorphic to a graph $H'$ with a vertex set contained in  $\binom{[n]}{k}\cup \binom{[n]}{k-1}$ such that   $V(H') \cap \binom{[n]}{k}$ is an edge-set of a $k$-partite $k$-uniform hypergraph.  We say that a graph has a {\it partite representation} if it has a $k$-partite representation for some $k$.
	Moreover, we call the map that brings $V(H)$ to  $V(H')$, a {\it $k$-partite embedding of $H$}.
	For example, if $H$ is an $8$-cycle, it has a $2$-partite representation with edges 
$12, 23, 34, 14$ corresponding to an $8$-cycle  with vertices $1, 12, 2, 23, 3, 34, 4, 14, 1$, in order. 
For a $k$-uniform hypergraph $H$, $\ex_k(t, H)$ denotes the largest number of edges in a $k$-uniform $t$-vertex hypergraph with no subgraph isomorphic to $H$.

\begin{theorem} [Conlon \cite{C}]\label{Conlon-rep-exponent}
Let $H$ be a cubical graph with $k$-partite representation $\cH$, for a fixed $k$.  If $\ex_k(t, \cH) \leq \alpha t^k$, then $\ex(Q_n, H) = O(\alpha^{1/k}n2^n)$. 
\end{theorem}

\begin{theorem}[Erd\H{o}s, \cite{E64}] \label{Erdos}
Let $k\geq 2$ be an integer and $K^{k}(\ell_1, \ldots, \ell_k)$ be the complete $k$-partite $k$-uniform hypergraph with parts of sizes $\ell_1, \ldots, \ell_k$.
Then $\ex_{k}(t, K^{k}(\ell_1, \ldots, \ell_{k})) = O(t^{k - 1/\delta})$, 
where  $\delta = \ell_1 \cdots \ell_{k-1}$. 
\end{theorem}

Theorem \ref{Erdos}
implies in particular, that  $\ex_k(t, \cH) \leq \alpha t^k$  for  $\alpha < t^{-a}$, for some positive $a$. Therefore, one can conclude the following fact about the Tur\'an density of graphs having partite representation.

\begin{corollary} 
If $H$ is a cubical graph  that has a  partite representation then $\ex(Q_n, H)= o(||Q_n||)$.
\end{corollary}

Note that having a partite representation is not a characterisation for graphs $H$ with $\ex(Q_n, H)= o(||Q_n||)$ as shown by the first author in \cite{MA1}.
For more recent results on such extremal hypergraph numbers, see Ma, Yuan, and  Zhang  \cite{MYZ}, as well as Mubayi and Verstra\"ete \cite{MV}.

\subsection{Subdivisions of cliques and bi-cliques}

For a graph $G$, we say that a graph $H$ is a $k$-{\it subdivision} of $G$ and denote it  $T_k(G)$ if $H$ is obtained from $G$ by ``inserting" $k$ vertices in each edge of $G$. Formally, 
$V(H)= V(G)\cup \bigcup_{e\in E(G)} V_e$, where $V(G)$ and $V_e$'s are pairwise disjoint, $|V_e|=k$ for each $e\in E(G)$,  and such that $G$ is a union of paths $P_e$ for $e\in E(G)$, where $P_e$ is a path on vertex set $\{x, y\}\cup V_e$ with endpoints $x$ and $y$, for $e=xy$. We shall call vertices from $V(G)$ {\it branch vertices},
paths $P_e$ {\it subdivision paths}, and vertices in $\bigcup_{e\in E(G)} V_e$ {\it subdivision vertices}. If $k$ is odd, we say that $T_k(G)$ is an {\it odd subdivision} of $G$, if 
$k$ is even, we say that $T_k(G)$ is an {\it even subdivision} of $G$.
Marquardt \cite{M} showed that $T_k(Q_n)$ has a partite representation for any odd $k$ and any $n$. Here we prove a more general result about an odd subdivision of any graph.

\begin{theorem}\label{complete}
For any integer $k\geq 0$ and any positive integer $t$, $T_{2k+1}(K_t)$ is  layered. Moreover,  for $k\geq 1$, 
$\ex(Q_n, T_{2k+1}(K_t)) = O(n^b2^n)= o(||Q_n||)$, where $b=1-\frac{1}{(k+1)t^k}$.
\end{theorem}

\begin{proof}
Let $G= T_{2k+1}(K_t)$. We shall be constructing an embedding of $G$ in $Q_n$, where the ground set $[n]$ is partitioned as follows:
 $$[n] =\bigcup_{x\in V(K_t)} A_x \cup \bigcup_{e\in E(K_t)} B_e,$$
where $B_e$'s and $A_x$'s are pairwise disjoint,  for each $x\in V(K_t)$ and $e\in E(K_t)$. 
For $k=0$, let    $A_x=\{x_1\}$ and  $B_e= \emptyset$, $e\in E(K_t)$.   For $k\geq 1$, let   $A_x=\{x_1, \ldots, x_k\}$  and  $B_e=\{b_e\}$, $e\in E(K_t)$.  So,  $n= tk+\binom{t}{2}$, for $k\geq 1$ and $n=t$ for $k=0$. \\

We shall define an embedding $f$ of $V(G)$ into $V(Q_n)$. 
Recall that  $V(Q_n)$ is the set of subsets of $[n]$.  If $x\in V(K_t)$, we also denote the respective branch vertex of $G$ by $x$.  Let $f(x) = A_x$. \\
 
 Now, consider two vertices $x, y\in V(K_t)$ forming an edge $e$. 
 Let the $xy$-subdividing path be $x, z_1, \ldots, z_{2k+1}, y$. \\
 
 If $k=0$, then $f(x) =\{x_1\}$, for any $x\in V(K_t)$,  let $f(z_1) = \{x_1, y_1\}$. \\

 If $k\geq 1$,  let $f(z_1) = \{b_e\} \cup f(x)$, $f(z_{2k+1})=  \{b_e\} \cup f(y)$. 
 For $1\leq i \leq k-1$, let   $f(z_{2i+1}) =  f(z_{2i-1}) - \{x_i\}\cup \{y_i\}$. 
 For $1\leq i\leq k$,  let  $f(z_{2i}) = f(z_{2i-1})- \{x_i\}$. The embedding $f$ is illustrated in Table \ref{Kn}.\\
 
\begin{table}[h]
\centering
\begin{tabular} {l||r|r|r||r|rr|rr||r|rrr|rrr} 
& \multicolumn{3}{c||}{$k=1$} & \multicolumn{5}{c||}{$k=2$} & \multicolumn{7}{c}{$k=3$}\\
&$~~~~b_e$& $x_1$& $y_1$	 							&$~~~~b_e$& $x_1$& $x_2$& $y_1$& $y_2$				&$~~~~b_e$& $x_1$& $x_2$& $x_3$& $y_1$&$y_2$&$y_3$\\
\hline \hline
$f(x)$& 	0&1&0												&0&1&1&0&0										&0&1&1&1&0&0&0\\
&&&&&&&&&&&&&&&\\
$f(z_1)$& 1&1&0												&1&1&1&0&0										&1&1&1&1&0&0 &0\\
$f(z_2)$& 1&0&0												&1&0&1&0&0										&1&0&1&1&0&0&0\\
$f(z_3)$& 1&0&1												&1&0&1&1&0										&1&0&1&1&1&0&0\\
$f(z_4)$& &&													&1&0&0&1&0										&1&0&0&1&1&0&0\\
$f(z_5)$& &&													&1&0&0&1&1										&1&0&0&1&1&1&0\\
$f(z_6)$& &&													&&&&&											&1&0&0&0&1&1&0\\
$f(z_7)$& &&													&&&&&											&1&0&0&0&1&1&1\\
&&&&&&&&&&&&&&&\\
$f(y)$& 	0&0&1												&0&0&0&1&1										&0&0&0&0&1&1&1
\end{tabular} 

\caption{Indicator vectors for $f(x), f(z_1), \ldots, f(z_{2k+1}), f(y)$, respectively, restricted to $B_e\cup A_{x}\cup A_y$, for $k=1, 2, $ and $3$.\label{Kn}}
\end{table}

This embedding is injective since distinct branch vertices are clearly mapped into distinct vertices of $Q_n$ not containing $b_e$ for any $e\in E(K_t)$.
On the other hand, any vertex subdividing an edge $e$ is mappped into one containing $b_e$, and not containing $b_{e'}$ for any $e'\in E(K_t)$, $e'\neq e$.
Thus a vertex subdividing $e$ and a vertex subdividing $e'$ for $e\neq e'$ are mapped into distinct vertices.\\

Finally, we see that the embedding is $(k+1)$-partite with parts $\{ x_1: ~ x\in V(K_t)\},$  $ \{x_2: ~ x\in V(K_t)\}$, $\ldots, $
$\{ x_k: ~ x\in V(K_t)\},$ and $\{b_e:~  e\in E(K_t)\}$, of sizes $t, t, \ldots, t, $ and $\binom{t}{2}$, respectively. By  Theorem \ref{Erdos} with $k+1$ instead of $k$, $\ell_1= \cdots = \ell_k = t$,  and  $\delta = \ell_1 \cdots \ell_{k} = t^k$,   $$\ex_{k+1}(n, K^{(k+1)}(\ell_1, \ldots, \ell_{k+1})) = O(n^{(k+1) - 1/\delta}) = O(\alpha n^{k+1}),$$ 
where  $\alpha =  n^{-t^{-k}}.$ 
Thus by Theorem \ref{Conlon-rep-exponent}, we have that  $$\ex(Q_n, G) =  O\left(\alpha^{\frac{1}{k+1}} n2^n\right)=  O\left(n^{-\frac{{t^{-k}}}{k+1}} n2^n\right) = o(||Q_n||).$$
\end{proof}

Next we consider even subdivisions. Since an even subdivision of an odd cycle is an odd cycle, that is not cubical, we only restrict ourselves to even subdivisions of bipartite graphs.  We shall consider even subdivisions of complete bipartite graphs. Note that  it is easy to see that $G=T_{2k}(K_{t,t})$ is cubical:  We shall consider an  embedding $f$ of  $G$ into $Q_n= Q_{2t+ 2k-1}$.  Let the parts of $G$ be ordered sets $A$ and $B$.  
Let $[n]= X\cup Y\cup Q$, where $X, Y, Q$ are pairwise disjoint sets, 
$X=\{x_1, \ldots, x_t\}$,  
$Y=\{y_1, \ldots, y_{2k-1}\}$, 
 $Q=\{q_1, \ldots, q_t\}$,  $|X|=|Q|=t$,  $|Y|=2k-1$.  

If $a\in A$ is  the  $i$th vertex from $A$, let $f(a) = X-\{x_i\}$. If $b\in B$ is the $j$th vertex from $B$, let $f(b) = \{q_j\} \cup Y\cup X$.
For the $i$th vertex of $A$, $a$,  and  for the $j$th vertex of $B$, $b$, let the $ab$-subdivision path be $a, z_1, z_2, \ldots, z_{2k}, b$, where $z_i=z_i(a,b)$, $i=1,\ldots, 2k$. 
Furthermore, let $f(z_1)= f(a) \cup \{q_j\}$,  $f(z_\ell) = f(z_{\ell-1})\cup \{y_{\ell -1}\}$, $\ell = 2, \ldots 2k$.  Note that $f(z_{2k}) = Y\cup \{q_j\} \cup X-\{x_i\}$.   The following theorem proves that $G$ is layered, which in particular implies that $G$ is cubical. However, the embedding presented in the theorem is a bit more involved. \\

\begin{theorem}\label{complete-bipartite}
For any positive integers $k$ and $t$, $T_{2k}(K_{t,t})$  is layered.
Moreover, for any even integer $k\geq 4$, and any positive integer $t$, 
$\ex(Q_n, T_{2k}(K_{t,t})) = O( n^{b}2^n) = o(||Q_n||)$, where $b= 1 -\frac{1}{2t(t^2+2)k^{t^2}}$.
\end{theorem}

\begin{proof}
Let $G= T_{2k}(K_{t,t})$. Let partite sets of $K_{t,t}$ be $A$ and $B$,  and respective  sets of branch vertices in $G$ also be $A$ and $B$.  We shall show that $G$ is layered for $k\geq 1$. In case when $k\geq 4$ and even, we show that it has zero Tur\'an density.\\

\noindent
{\bf Case $k=1$.~}  We shall embed $G$ into $Q_n= Q_{2t+1}$ using the embedding $f$ as follows. Let $[n]= X\cup Y\cup \{q\}$, $X=\{x_1, \ldots, x_t\}$, 
$Y=\{y_1, \ldots, y_{t}\}$, where $X$, $Y$ and $\{q\}$ are pairwise disjoint. For the $i$th vertex  $a$ in $A$, let $f(a) = \{x_i\} \cup Y$.
For the $j$th vertex $b$ of $B$, let $f(b) = Y-\{y_j\} \cup \{q\}$. 
For the $ab$-subdivision path $a, z_1, z_2, b$, let $f(z_1) = f(a) - \{y_j\}$ and $f(z_2) = f(z_1) \cup \{q\}$. 
Then this is an embedding in layers $t$ and $t+1$. \\

\noindent
{\bf Case $k=2$.~}  We shall embed $G$ into $Q_n= Q_{2t+3}$ using the embedding $f$ as follows. Let $[n]= X\cup Y\cup \{q_1, q_2, q_3\}$, $X=\{x_1, \ldots, x_t\}$, 
$Y=\{y_1, \ldots, y_{t}\}$, where $X$, $Y$ and $\{q_1, q_2, q_3\}$ are pairwise disjoint. For the $i$th vertex  $a$ in $A$, let $f(a) = \{x_i\} \cup Y \cup \{q_3\}$.
For the $j$th vertex $b$ of $B$, let $f(b) = Y-\{y_j\} \cup \{q_1, q_2\}$. 
For the $ab$-subdivision path $a, z_1, z_2, z_3, z_4, b$, let $f(z_1) = f(a) - \{y_j\}$, $f(z_2) = f(z_1) \cup \{q_1\}$, 
$f(z_3) = f(z_2) - \{q_3\}$, and $f(z_4) = f(z_3) \cup \{q_2\}$. 
This is an embedding in layers $t+1$ and $t+2$.\\

This embedding is injective since for  any subdivision vertex $z$ of the edge $ab$ of $K_{t,t}$, where $a$ is the $i$th and $b$ is the $j$th vertex of respective parts $A$ and $B$, it must be the case that  $f(z)\cap (X\cup Y) = \{x_i\} \cup Y- \{y_j\}$. So subdivision vertices for distinct edges are mapped into distinct vertices.  Other pairs of distinct vertices of $G$ are mapped to distinct vertices as witnessed by $A$, $B$, or $\{q_1, q_2, q_3\}$. \\

\noindent 
{\bf Case $k\geq 3$.~ } We shall show that  $G$ is embeddable in a layer and for even $k\ge 4$, $G$ has a partite representation.
 Let $n= 2t+1+ t^2(k-1)$.  
Let $$[n] =A \cup B  \cup  \{c\} \cup \bigcup_{e\in E(K_{t,t})} S_e,$$
where $A$, $B$, $\{c\}$,  and $S_e$'s are all pairwise disjoint,  and for any $e\in E(K_{t,t})$,   $S_e=\{s_e^1, s_e^2, \ldots, s_e^{k-1}\}$, $|S_e|=k-1$.\\

We shall define an embedding $f$ of $V(G)$ into $V(Q_n)$. 
Recall that vertices of $Q_n$  are subsets of $[n]$.   Consider first the branch vertices. Let $S= \{s_e^1: ~ e\in E(K_{t,t})\}$.  For any $a\in A$, let $f(a) = \{a,c \}\cup S$ and  for  any $b\in B$, let $f(b) = \{b\}\cup S$. \\

 Now, consider two vertices $a, b\in V(K_{t,t})$ forming an edge $e$. 
 Let the $ab$-subdividing path be $a, z_1, \ldots, z_{2k}, b $.  
 Let $S'_e= \{s_{e'}^1: ~ e'\neq e\}$. 
 Then we see that $f(a) =\{a,c \}\cup S= \{a, c,  s_e^1 \} \cup S'_e$ and $f(b) =\{b\}\cup S= \{b, s_e^1\}\cup S'_e$. 
 Let 
 \begin{eqnarray*}
  f(z_1) = \{a, c \} \cup S'_e,&& f(z_2) = \{a, c,  s_e^2 \} \cup S'_e\\
  f(z_3) = \{  c, s_e^2\} \cup S'_e,&& f(z_4) = \{b, c, s_e^2\} \cup S'_e \\
  f(z_5) = \{b, s_e^2 \} \cup  S'_e, && f(z_{2k}) = f(z_{2k-1})\cup \{s_e^1\}.
 \end{eqnarray*}
 
 Furthermore, for $k\geq 4$ and $ 1\leq i \leq k-3$, let 
 $$f(z_{6+2i-2} ) = f(z_{6+2i-3}) \cup \{s_e^{2+i}\}~~ \mbox{ and}   ~~f(z_{6+2i-1} ) = f(z_{6+2i-2}) -\{s_e^{1+i}\}.$$

This is an embedding into layers $t^2$ and $t^2+1$ because each $S’_e$ has size $t^2-1$.
In Table \ref{Knn}, we illustrate this embedding.

\begin{table}
\centering
\begin{tabular} {l||rrr|rr||rrr|rrr||rrr|rrrr} 
& \multicolumn{5}{c||}{$k=3$} & \multicolumn{6}{c||}{$k=4$} & \multicolumn{7}{c}{$k=5$}\\
&$~~~~a$& $b$& $c$& $s_e^1$& $s_e^2$					&$~~~~a$& $b$& $c$& $s_e^1$& $s_e^2$&$s_e^3$			&$~~~~a$& $b$& $c$& $s_e^1$& $s_e^2$&$s_e^3$&$s_e^4$\\
\hline \hline
$f(a)$& 1&0&1&1&0											&1&0&1&1&0&0													&1&0&1&1&0&0&0\\
&&&&& &&&&&& &&&&&&&\\
$f(z_1)$& 1&0&1&0&0										&1&0&1&0&0&0													&1&0&1&0&0&0 &0\\
$f(z_2)$& 1&0&1&0&1										&1&0&1&0&1&0													&1&0&1&0&1&0&0\\
$f(z_3)$& {\bf 0}&{\bf 0} &{\bf 1} & {\bf 0} &{\bf 1}	&{\bf 0}&{\bf 0} &{\bf 1} & {\bf 0} &{\bf 1} & {\bf 0}	&{\bf 0}&{\bf 0} &{\bf 1} & {\bf 0} &{\bf 1} & {\bf 0} &{\bf 0}\\
$f(z_4)$& {\bf 0}&{\bf 1} &{\bf 1} & {\bf 0} &{\bf 1}	&{\bf 0}&{\bf 1} &{\bf 1} & {\bf 0} &{\bf 1}& {\bf 0}		&{\bf 0}&{\bf 1} &{\bf 1} & {\bf 0} &{\bf 1}& {\bf 0}&{\bf 0}\\
$f(z_5)$& {\bf 0}&{\bf 1} &{\bf 0} & {\bf 0} &{\bf 1}	&{\bf 0}&{\bf 1} &{\bf 0} & {\bf 0} &{\bf 1}& {\bf 0}		&{\bf 0}&{\bf 1} &{\bf 0} & {\bf 0} &{\bf 1}& {\bf 0}&{\bf 0} \\
$f(z_6)$& 0&1&0&1&1										&0&1&0&0&1&1													&0&1&0&0&1&1&0\\
$f(z_7)$& &&&&												&0&1&0&0& 0&1													&0&1&0&0& 0&1&0\\
$f(z_8)$& &&&&												&0&1&0& 1& 0&1												&0&1&0&0& 0&1&1\\
$f(z_9)$& &&&&												& &&&&&															&0&1&0&0& 0&0&1\\
$f(z_{10})$& &&&&											& &&&&&															&0&1&0& 1& 0&0&1\\
&&&&& &&&&&& &&&&&&&\\
$f(b)$& 0&1&0&1&0											&0&1&0& 1& 0&0												&0&1&0& 1& 0&0&0
\end{tabular}

\caption{Indicator vectors for $f(a), f(z_1), \ldots, f(z_{2k}), f(b)$, respectively, restricted to $a, b, c, s_e^1, s_e^2,  \ldots, s_e^{k-1}$ in order, for $k=3, 4, $ and $5$.\label{Knn}}
\end{table}
\vskip 1cm

This embedding is injective since distinct branch vertices are clearly mapped into distinct vertices of $Q_n$. Moreover, any branch vertex $x$ and any subdivision vertex $z$ are mapped to different vertices by $f$ because $f(x) \cap \bigcup_{e\in E(K_{t,t})} S_e = S$ and $f(z)  \cap \bigcup_{e\in E(K_{t,t})} S_e \neq S$.
For any two vertices $z, z'$ subdividing an edge $e$, it is clear from the definition that $f(z)\neq f(z')$. 
Finally for a vertex $z$ subdividing an edge $e$ and vertex $z'$ subdividing an edge $e'$,  $e\neq e'$, $f(z) \cap S_e \neq f(z') \cap S_{e'}$.\\

We see that, for even $k\geq 4$,  this embedding gives a partite representation with parts
\begin{itemize}
\item{}
$A\cup B$, 
\item{} $\{s_e^1, s_e^2, s_e^4, s_e^6, \ldots \}$, $e\in E(K_{t,t})$, and 
\item{} 
$\{c\} \cup   \bigcup_{e\in E(K_{t,t})} \{s_e^3, s_e^5, s_e^7, \ldots\}.$
\end{itemize}

The sizes  $\ell_1, \ldots, \ell_{t^2+2}$ of the parts are at most $2t, k, k, \ldots, k, kt^2$, respectively.
By  Theorem \ref{Erdos}, with $q= t^2+2$ instead of $k$, and $\delta = \ell_1 \cdots \ell_{q-1} \leq 2t\cdot k^{t^2}$,   $$\ex_{q}(n, K^{q}(\ell_1, \ldots, \ell_{q})) = O\left(n^{q - \frac{1}{\delta}}\right) = O(\alpha n^{q}),$$ 
where  $\alpha =  n^{-1/\delta}.$ Note that $\alpha^{\frac{1}{q}}=n^{-\frac{{1}}{\lambda}}$ for $\lambda=2t(t^2+2)k^{t^2}$.
Thus by Theorem \ref{Conlon-rep-exponent}, we have that  $$\ex(Q_n, G) =  O\left(\alpha^{\frac{1}{q}} n2^n\right)=  O\left(n^{-\frac{{1}}{\lambda}} n2^n\right) = o(||Q_n||).$$
\end{proof}

\section{Layered embedding of  theta graphs, non-layered graphs of girth eight, proof of Theorem \ref{girth8}}	\label{girth}

	A graph is a {\it theta graph}  with legs of length $\ell_1, \ldots, \ell_k$  and {\it poles} $v$ and $v'$   if it is a union of $k$ paths of lengths $\ell_1, \ldots, \ell_k$  with endpoints $v$ and $v'$  whose vertex sets pairwise share only $\{v, v'\}$. Here, we shall denote the Hamming distance between two sets or two binary sequences  $x,y$ as $d_H(x,y)$. 
	Note that $C_4$  is a theta graph with two legs of length $2$ and it is not a layered graph.

	\begin{lemma}
		If  $G$ is a theta graph with arbitrary number of legs of length $m\geq 3$ each, then $G$ is cubical.
		If $G$ is a theta graph with $3$ legs of length $2$ each, i.e., $G=K_{2,3}$, then $G$ is not cubical.
		\end{lemma}
	
	\begin{proof}
	We shall define an edge-coloring of $G$ as follows. 
	Let the edges of the $i$th  leg  incident to the poles be colored $i$, $i=1, \ldots, m$. 
	Let  all edges at distance $k$ from the first pole be colored $x_k$,  $k=1, \ldots, m-2$, for distinct $x_1, \ldots, x_{m-2}$
	different from any of $1, \ldots, m$.
Then this coloring satisfies the 	properties of Havel-Moravek, \cite{HM}.
To prove the second statement of the lemma, observe that a nice coloring must assign colors $1, 2, 1, 2$ to the edges of any $C_4$ up to renaming the colors. This is impossible to maintain in a $K_{2,3}$.
 \end{proof}

	 	\begin{lemma}\label{theta-symmetric}
	Let $G$ be a theta graph with poles $a$ and $a'$ and $t$ legs of length $m$ each, $t>\lceil \frac{m}{2} \rceil$.  If $G$ is a subgraph of a layer, then  $d_H(a, a')<m$. 
\end{lemma}
	 
	 \begin{proof}
	 Assume that $d_H(a, a')\geq m$. Since there is a path of length $m$ between $a$ and $a'$, $d_H(a, a')=m$.
	 Assume without loss of generality that $a$ is in a lower or the same layer as $a'$.
	 Let $S$ be the set of $m$ coordinates where $a$ and $a'$ differ. Since the number of $0$'s in $a$ and $a'$ differ by at most one,  $a$ has at most $\lceil \frac{m}{2} \rceil$ zeros in positions from $S$.
	 Then for  any $a, a'$-path $P$ of length $m$, and any $s\in S$,  there should be an edge with a star with position in $s$.
	 Thus, each edge  of $P$ has stars only in positions from $S$. 
	 Moreover, a first edge of $P$ can have stars only in positions corresponding to $0$'s of $a$. Hence,
	 there are at most $\lceil \frac{m}{2} \rceil$ such edges. Thus, $t\leq \lceil \frac{m}{2} \rceil$.
		 \end{proof}
	
	\begin{lemma}
	Let $G$ be a theta graph  with poles $a$ and $a'$ with $3$ legs of length $3$ each. Then $G$ is not embeddable in a layer.
	\end{lemma}

	\begin{proof}
Assume that  $G$ is  layered. By Lemma \ref{theta-symmetric} we have that $d_H(a, a') <3$. Since $a$ and $a'$ are in different vertex layers, $d_H(a, a') =1$. Then the edge $aa'$ and one of the legs of $G$ form a $C_4$,  a contradiction since $C_4$ is not a layered graph.
	\end{proof}

Let $G_8'$ be a theta graph with poles $a$ and $a'$, three legs of length $4$ each, a vertex $u$ adjacent to $a$ on one leg and a vertex $u'$ adjacent to $a'$ on another leg. 
Let $G_8$ be a union of $G_8'$ and a $u,u'$-path of length $4$ internally disjoint from $G_8'$. See Figure \ref{fig:cubicalG8} (i).

\begin{figure}[h]
\centering
\includegraphics[scale=0.6]{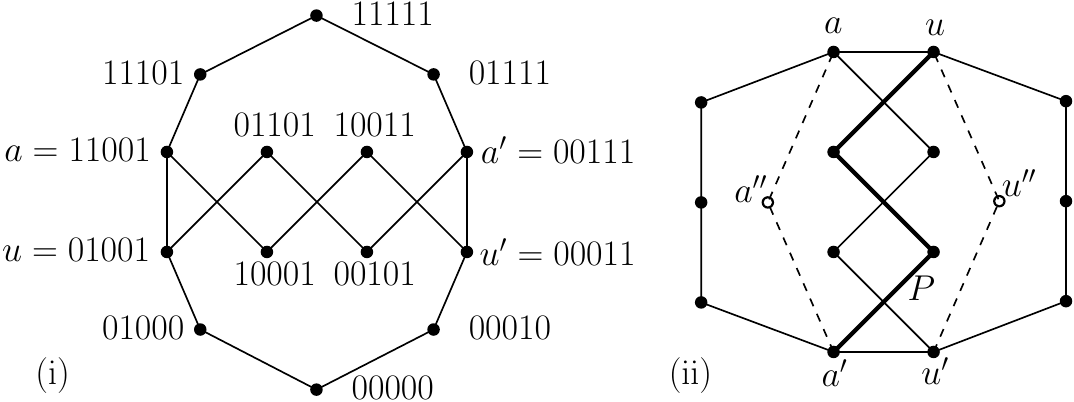}
\caption{(i) Embedding of $G_8$ into $Q_5$, \quad (ii) the layered graph $G_8\cup P_a \cup P_u$}
\label{fig:cubicalG8}
\end{figure}

Now, Theorem \ref{girth8} follows immediately from the following lemma.

\begin{lemma}
The graph $G_8$ is cubical, girth $8$ and not layered.
\end{lemma}

\begin{proof}
To see that $G_8$ is cubical, see an embedding in $Q_5$ shown in  Figure \ref{fig:cubicalG8} (i).  
From now on, we assume that $G_8$ is embedded into two vertex layers $V_\ell$ and $V_{\ell+1}$ of $Q_n$ for some $n$. 
Let $X=V_\ell\cup V_{\ell+1}$.\\

By Lemma \ref{theta-symmetric} and using parity, $d_H(u,u')=d_H(a,a')=2$ and thus in particular, $u$ and $u'$ are in the same layer of $Q_n$ and 
$a$ and $a'$ are in the same layer of $Q_n$. 
Consider  internal vertices  $a''$, $u''$ on a shortest $a,a'$-path $P_a$ and $u,u'$-path $P_u$  in $Q_n$, respectively, such that $a'', u''\in X$.  We see that $a''\neq u''$ and  $a'' , u'' \not\in V(G_8)$ since otherwise $G_8$ contains a $4$-cycle. 
Then $G_8\cup P_a \cup P_u$ is embedded in $X$ and thus has a very nice coloring $c$.
We know that the edges in any $6$-cycle in $Q_n$ have exactly three directions giving  a very nice coloring, $123123$, up to renaming the colors.
In particular, for the $6$-cycle $uaa''a'u'u''u$ we have that $c(ua)=c(u'a')$. 
However, there is a path $P$ of length $3$ between $ua$ and $u'a'$, see Figure \ref{fig:cubicalG8} (ii), contradicting the fact that $c$ is very nice. 
\end{proof}

\section{Lower bounds for $\ex(Q_n, C_{10})$, proof of Theorem \ref{thm:C10}}\label{sec:C10}

In order to obtain a lower bound on  $\ex(Q_n, C_{10})$, we shall make a construction that uses a construction by Conder   \cite{conder} several times.   Conder's construction uses an edge-coloring of $Q_n$ that we call a {\it prefix coloring}. 
Recall that we represent an edge $AB$, $A\subseteq B$ in $Q_n$ by a sequence of length $n$, 
where the $i$th position is occupied by $0$ if $i\notin B$,  by $1$ if $i\in A$ and by $\star$ if $i\in B\setminus A$. 
If $e$ is a  vector of $1$s, $0$s, and a star,  let  $\pre(e)$ be the the number of $1$'s in the positions of $e$ preceding the star position and  $\suf(e)$ be  the number of $1$'s in the positions of $e$ following the star position. 
Let $$f(e)= \pre(e)-\suf(e) \pmod 3.$$  Then $f$ is called the {\it prefix coloring} of $e$.  For example $f(01001\star 01) = 2 - 1 =1 \mod 3 $.

\begin{proof}[Proof of Theorem \ref{thm:C10}]
We are to prove that $\ex(Q_n, C_{10}) = \Omega\left( \frac{n2^n}{\log^a n} \right)$, where $a = \log_2 3$.
Let $\pi$ be a permutation of $[n]$.  For an edge $e=(x_1, \ldots, x_n)$ of $Q_n$ given in star representation, we let $e_\pi$ be the representation of $e$ with respect to $\pi$, that is, a vector $(x_{\pi(1)}, \ldots, x_{\pi(n)})$.
Let $\Pi$ be a smallest  set of permutations of $[n]$ such that for any ordered set $(a, b, c)$,  with distinct elements $a, b, c \in [n]$,  there is $\pi \in \Pi$ such that $\pi^{-1}(a)< \pi^{-1}(b)$ and $\pi^{-1}(a)< \pi^{-1}(c)$.
By a result of Spencer \cite{S}, there exists such a set with $|\Pi|\leq  \log_2\log_2 n$. Let $\Pi = \{\pi_1,  \pi_2, \ldots \}$.  \\

We shall define an edge-coloring $g$ of $E(Q_n)$ as follows. 
Let, for $e\in E(Q_n)$, $$g(e) = \left(g_0(e), g_{\pi_1}(e), \ldots, g_{\pi_{|\Pi|}}(e)\right),$$ where for any $\pi\in \Pi$ we have that
$g_\pi(e)$ is  a prefix coloring, i.e.,  $g_\pi(e)=f(e_{\pi})$ and $g_0(e)$ is equal to the parity of a layer containing $e$, i.e., $g_0(e)$ is $0$ if $e$ is in an even layer and it is $1$ if $e$ is in an odd layer.  \\

Since each prefix coloring uses exactly three colors, the total number of colors used by $g$ is $2\cdot 3^{|\Pi|}$. 
We shall argue that there is no monochromatic $C_{10}$ under this coloring. Then taking a largest color class of $g$, we obtain a desired $C_{10}$-free subgraph of $Q_n$. \\

Consider a copy $C$ of $C_{10}$ in $Q_n$. 
If $C$ is monochromatic under $g$, it is in particular monochromatic under $g_0$. Note that since $C$ is connected, and $g_0$ distinguishes even and odd edge layers, $C$ must be contained in some edge layer.
Consider a very nice coloring  $\eta$ of $C$ corresponding to the directions of its edges. Each color in $\eta$ must  appear an even number of times.
If there is a color in $\eta$ that appears $4$ times, there are two edges of that color that are at distance at most $1$, contradicting the fact that $\eta$ is very nice. 
Thus each color in $\eta$ appears exactly twice. Let these colors be $a, b, c, d, e$. 
I.e.,  $C$ has exactly $5$ star positions on its edges and  these positions are $a, b, c, d, e $.  All vertices of $C$ coincide on all other positions. \\

Consider a hypergraph $H_C$ on the vertex set $\{a, b, c, d, e\}$, whose hyperedges correspond to non-zero positions of edges of $C$, restricted to $\{a, b, c, d, e\}$. 
Because the number of $1$s in every edge must be the same, $H_C$ is a uniform hypergraph of uniformity $2, 3, $ or $4$, such that in some of the edge-orderings, considering intersections of consecutive edges, gives us $5$ distinct sets.
As can be seen by inspection, the possible such hypergraphs  $H_C$,  up to a permutation of $\{a, b, c, d, e\}$ are 
$H_1=\{ab, bc, cd, de, ea\}$,  $H_2=\{abc, bcd, cde, dea, eab\}$,  $H_3=\{cde, dea, aeb, ebc, bcd\}$, $H_4=\{abc, bcd, cde, bde, bda\}$,  or  
$H_5=\{abcd, bcde, cdea, deab, eabc\}$, as shown in Figure \ref{fig:HC}.  See the respective edges of $H_1, H_2, H_3$ and $H_4$ in the list below. Note that $H_5$ is very similar to $H_1$ and is obtained by switching $1$'s and $0$'s in the star representations of the edges. We shall thus not consider $H_5$.\\

\begin{figure}[h]
\centering
\includegraphics[scale=0.55]{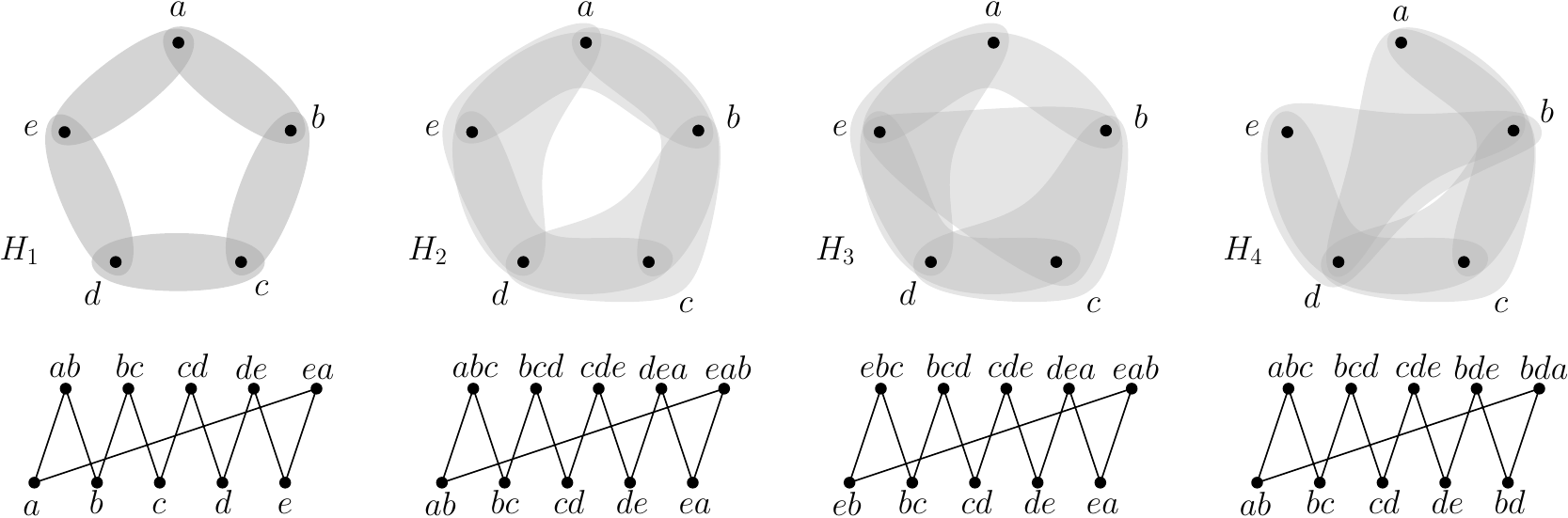}
\caption{Hypergraphs $H_1$, $H_2$, $H_3$ and $H_4$ with corresponding $10$-cycle $C$.}
\label{fig:HC}
\end{figure}

\begin{table}
\begin{tabular}{rrrrr|}
 a& b & c& d & e\\
\hline
 1 & \st & 0& 0& 0\\
 \st & 1& 0 & 0& 0\\
0 & 1& \st& 0& 0\\
0 & \st& 1& 0& 0\\
0 & 0  & 1& \st& 0\\
0 & 0& \st & 1 & 0\\
0 & 0& 0 & 1& \st\\
0 & 0& 0& \st & 1\\
\st   & 0& 0& 0& 1\\
1 & 0& 0& 0& \st\\
$H_1$&& &&
\end{tabular}
\hskip 0.2cm
\begin{tabular}{rrrrrr|}
& a& b & c& d & e\\
\hline
& 1 & 1& \st& 0& 0\\
$e_1^a$ & \st & 1& 1& 0& 0\\
&0 & 1& 1& \st& 0\\
&0 & \st& 1& 1& 0\\
&0 & 0&  1& 1& \st\\
&0 & 0& \st & 1 & 1\\
$e_2^a$ &\st & 0& 0 & 1& 1\\
&1 & 0& 0& \st & 1\\
&1  & \st& 0& 0& 1\\
&1 & 1& 0& 0& \st\\
$H_2$ &&&&&
\end{tabular}
\hskip 0.2cm
\begin{tabular}{rrrrr|}
~ a& b & c& d & e\\
\hline
 0& 0&  1& 1& \st\\
 0 & 0& \st & 1& 1\\
\st & 0& 0&  1 & 1\\
1 & 0& 0& \st& 1\\
1 & \st& 0& 0 & 1\\
\st & 1& 0 & 0& 1\\
0 & 1& \st& 0& 1\\
0  & 1& 1& 0 & \st \\
0 & 1& 1& \st& 0\\
0 & \st& 1& 1& 0\\
$H_3$ &&&&
\end{tabular}
\hskip 0.2cm
\begin{tabular}{rrrrr}
~ a & e & c & d & b\\
\hline
 1 & 0& \st& 0& 1\\
 \st & 0& 1 & 0& 1\\
0 & 0& 1& \st& 1\\
0 & 0& 1& 1& \st\\
0 & \st& 1& 1 & 0\\
0 & 1& \st & 1& 0\\
0 & 1& 0& 1& \st \\
0& \st& 0& 1& 1\\
\st & 0& 0& 1& 1\\
1 & 0& 0& \st& 1\\
$H_4$ &&&&
\end{tabular}

\caption{Star representation of edges in $C$ where $H_C=H_i$, $i\in[4]$.\label{table:Hi}}
\end{table}

We show that $C$ is not monochromatic under some $g_\pi$, $\pi\in\Pi$. 
We distinguish four cases depending on $i\in[4]$ with $H_C=H_i$. Note that in the first three cases we actually prove that $C$ is not monochromatic under any $g_\pi$. 
Fix an arbitrary permutation $\pi\in\Pi$ and assume that $C$ is monochromatic under $g_\pi$. \\

For any two distinct $x,y\in\{a,b,c,d,e\}$ we define an indicator function
$\overrightarrow{xy}= \overrightarrow{xy}_\pi$ that is equal to $1$ if $x$ appears before $y$ in $\pi$, and it is equal to $-1$ otherwise. 
Note that $\overrightarrow{xy}=-\overrightarrow{yx}$, in particular $\overrightarrow{xy}\neq\overrightarrow{yx}$. 
In the remainder of the proof all equations are considered modulo $3$.
Let $e_1^x$ and $e_2^x$ be any two edges in $C$ with star position at the same $x\in\{a,b,c,d,e\}$.
Then let  $X_j=\big\{y\in\{a,b,c,d,e\}: ~ e_j^x$ has a $1$ at position $y\big\}$ for $j\in[2]$.
Note that the prefix coloring of $e_j^x$ with respect to permutation $\pi$ and restricted to $\{a,b,c,d,e\}$ is equal to the sum $\sum_{y\in X_j} \overrightarrow{yx}_\pi$. In every position other than $\{a,b,c,d,e\}$ both $e_1^x$ and $e_2^x$ coincide. Thus 
if $e_1^x$ and $e_2^x$ have the same color, then
\begin{equation}
\sum_{y\in X_1} \overrightarrow{yx}=\sum_{y\in X_2} \overrightarrow{yx}. \tag{$*$}
\end{equation}
For example consider the two edges $e_1^a$ and $e_2^a$ in $C$ with star position at $a$, where $H_C=H_2$, as indicated in Table \ref{table:Hi}.
Then $X_1=\{b,c\}$ and $X_2=\{d,e\}$ and the edges differ only in positions $X_1\cup X_2$, therefore $e_1^a$ and $e_2^a$ have the same color only if $\overrightarrow{ba}+\overrightarrow{ca}=\overrightarrow{da}+\overrightarrow{ea}$. 
\\

{\bf Case 1.  } $H_C=H_1$.\\
Applying $(*)$ for the two edges of $C$ with star position at $a$ provides that $\overrightarrow{ba}=\overrightarrow{ea}$.
  
Similarly, considering edges with star positions at $b, c, d$, and $e$, we have 
$\overrightarrow{ab}=\overrightarrow{cb}, ~
\overrightarrow{bc}=\overrightarrow{dc}, ~
\overrightarrow{cd}=\overrightarrow{ed}, ~
\overrightarrow{de}=\overrightarrow{ae}.$
Then $\overrightarrow{ba}=\overrightarrow{ea}=-\overrightarrow{de}=\overrightarrow{cd}=-\overrightarrow{bc}=\overrightarrow{ab}$. This is a contradiction.
\\

{\bf Case 2.  } $H_C=H_2$.\\

Applying $(*)$ for pairs of edges with star positions in $\{a,b,c,d,e\}$ we obtain five equations, which we then add up:
$$\begin{cases}
\begin{array}{c}
\overrightarrow{ba}+\overrightarrow{ca}=\overrightarrow{da}+\overrightarrow{ea}\\
\overrightarrow{cb}+\overrightarrow{db}=\overrightarrow{eb}+\overrightarrow{ab}\\
\overrightarrow{ac}+\overrightarrow{bc}=\overrightarrow{dc}+\overrightarrow{ec}\\
\overrightarrow{bd}+\overrightarrow{cd}=\overrightarrow{ed}+\overrightarrow{ad}\\
\overrightarrow{ae}+\overrightarrow{be}=\overrightarrow{ce}+\overrightarrow{de}
\end{array}
\end{cases}
\Longrightarrow
\begin{cases}
\begin{array}{c}
\overrightarrow{ba}+\overrightarrow{cd}+\overrightarrow{ae}+\overrightarrow{be}=\overrightarrow{ea}+\overrightarrow{eb}+\overrightarrow{ab}+\overrightarrow{dc}.
\end{array}
\end{cases}$$

As a result $0=\overrightarrow{ea}+\overrightarrow{eb}+\overrightarrow{ab}+\overrightarrow{dc}$. 
If $\overrightarrow{ea} = \overrightarrow{ab}$, the transitivity in $\pi$ implies that $\overrightarrow{eb}=\overrightarrow{ea} = \overrightarrow{ab}$.
Recall that all equations are considered modulo $3$.
Therefore, $\overrightarrow{ea}+\overrightarrow{eb}+\overrightarrow{ab}=0$, so $\overrightarrow{dc}=0$, which is a contradiction.
Thus $\overrightarrow{ae} = \overrightarrow{ab}$. By a symmetric argument $\overrightarrow{ab}=\overrightarrow{cb}, ~
\overrightarrow{bc}=\overrightarrow{dc}, ~
\overrightarrow{cd}=\overrightarrow{ed}, ~
\overrightarrow{de}=\overrightarrow{ae}.$
This is exactly the condition obtained in Case 1, a contradiction.\\

{\bf Case 3.  } $H_C=H_3$.\\
In this case, again consider pairs of edges with star positions at $a$, $b$, etc. and for each pair set up an equation:
$$\begin{cases}
\begin{array}{c}
\overrightarrow{ea}+\overrightarrow{ba}=\overrightarrow{da}+\overrightarrow{ea}\\
\overrightarrow{cb}+\overrightarrow{db}=\overrightarrow{eb}+\overrightarrow{ab}\\
\overrightarrow{dc}+\overrightarrow{ec}=\overrightarrow{bc}+\overrightarrow{ec}\\
\overrightarrow{bd}+\overrightarrow{cd}=\overrightarrow{ad}+\overrightarrow{ed}\\
\overrightarrow{ce}+\overrightarrow{be}=\overrightarrow{ce}+\overrightarrow{de}
\end{array}
\end{cases}
\Longrightarrow
\begin{cases}
\begin{array}{c}
\overrightarrow{ba}=\overrightarrow{da}\\
\overrightarrow{bc}+\overrightarrow{bd}=\overrightarrow{be}+\overrightarrow{ba}\\
\overrightarrow{dc}=\overrightarrow{bc}\\
\overrightarrow{db}+\overrightarrow{dc}=\overrightarrow{da}+\overrightarrow{de}\\
\overrightarrow{be}=\overrightarrow{de}
\end{array}
\end{cases}
\Longrightarrow
\begin{cases}
\begin{array}{c}
\overrightarrow{dc}+\overrightarrow{bd}=\overrightarrow{de}+\overrightarrow{da}\\
\overrightarrow{db}+\overrightarrow{dc}=\overrightarrow{da}+\overrightarrow{de}.
\end{array}
\end{cases}$$
Thus $\overrightarrow{bd}=\overrightarrow{db}$, which is a contradiction.
\\

{\bf Case 4.  } $H_C=H_4$.\\
In each of the prior cases, we showed that $C$ cannot be monochromatic under any $\pi\in\Pi$. In this case, there may be exist a permutation $\pi\in\Pi$ such that $C$ is monochromatic under $g_{\pi}$, see for example, the order of $a,b,c,d,e$ as in Table \ref{table:Hi}. 
Thus for $H_4$, we will only show that $C$ is not monochromatic for some $g_{\pi'}$, $\pi'\in\Pi$.
Assume that $C$ is monochromatic under $g$. Then for any permutation $\pi\in\Pi$, considering pairs of edges with star positions at  $b$ and $e$, the following two equations hold:
 $$
 \begin{cases}
\begin{array}{c}
\overrightarrow{cb}+\overrightarrow{db}=\overrightarrow{eb}+\overrightarrow{db}\\
\overrightarrow{ce}+\overrightarrow{de}=\overrightarrow{be}+\overrightarrow{de}\\
\end{array}
\end{cases}
\Longrightarrow
\begin{cases}
\begin{array}{c}
\overrightarrow{cb}=\overrightarrow{eb}=\overrightarrow{ec}.
\end{array}
\end{cases}$$

Thus  $c$ is between   $e$ and $b$. 
By the way $\Pi$ was selected, there is some permutation $\pi' \in \Pi$ such that $c$ precedes both $b$ and $e$, a contradiction. 
\\

Thus, in at least some coloring $g_\pi$, $\pi\in \Pi$, $C$ is not monochromatic.
Therefore, $C$ is not monochromatic under the coloring $g$. 
The number of colors used by $g$ is $2\cdot 3^{|\Pi|} \leq 2 \cdot 3^{\log_2\log_2 n} = 2\cdot (\log_2 n)^{\log_2 3}$.
Consider a largest color class of $g$ having $n2^{n-1}/ 2\cdot (\log_2 n)^{\log_2 3}$ edges. It contains no copy of $C_{10}$. 
\end{proof}

	Next, we remark that there is a lower bound on $\ex(Q_n, C_{10})$ using a special extremal function for a smaller graph. 
	Let $\ex^*(Q_n, C_6^-)$ be the largest number of edges in a subgraph $G$ of $Q_n$ such that $G$ contains no $C_6^-$, that is a subgraph  $H$ of $Q_n$ on $6$ vertices and $5$ edges such that $H$ is a subgraph of $C_6$ in $Q_n$. Note, that $C_6^-$ forms a path of length $5$, but not every path of length $5$ is a $C_6^-$. For example, the path $00000, 
00001, 00011, 00111, 01111, 11111$ is not a $C_6^-$ because its edges do not form a subgraph of a $6$-cycle.\\

\begin{lemma}\label{lem:C6-}
$\ex^*(Q_n, C_6^-)/4 \leq \ex(Q_n, C_{10})$. 
\end{lemma}

\begin{proof}
Observe first that any non-induced copy  of $C_{10}$  in $Q_n$ contains $C_6^-$ as a subgraph.
Let $G$ be a largest $C_6^-$-free subgraph of $Q_n$.  It does not contain a copy of $C_{10}$ that is non-induced in $Q_n$. Let $c$ be a coloring of $E(Q_n)$ in $4$ colors with no monochromatic induced copy of $C_{10}$,  see \cite{AM}. Let $G'$ be a largest monochromatic subgraph of $G$ under $c$. We have that $G'$ has no copy of $C_{10}$ that is induced in $Q_n$ and $||G'||\geq ||G||/4= \ex^*(Q_n, C_6^-)/4$.  Thus $G'$  has neither induced, nor non-induced copies of $C_{10}$ and thus it is $C_{10}$-free.  
\end{proof}

We have an exact value for $\ex^*(Q_n, C_6^-)$ if $n=3$:

\begin{lemma}
$\ex^*(Q_3, C_6^-)=8$.
\end{lemma}

\begin{proof}
To see that $\ex^*(Q_3, C_6^-)\geq 8$ consider a subgraph of $Q_3$ that is a vertex disjoint union of two $C_4$'s.  
For the upper bound it suffices to show that $\ex(Q_3, C_6)\leq 9$.
Let $G$ be a subgraph of $Q_3$ on $10$ edges. It is easy to check that in each configuration of the two non-present edges, there is a $6$-cycle, see Figure \ref{fig:exC6}.
\begin{figure}[h]
\centering
\includegraphics[scale=0.6]{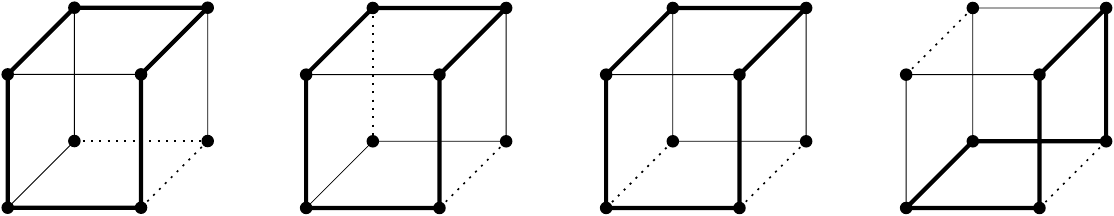}
\caption{$6$-cycle in $Q_3$ in each configuration of non-present edges (dotted).}
\label{fig:exC6}
\end{figure}
\end{proof}

\section{Density of layered graphs}\label{density} 

In this section we prove some results about density of layered graphs. We show that under classical compression operation the density of a layered graph is not decreased. Moreover, if the compressed graph corresponds to initial intervals in colex order, we can show  that the asymptotic density of the graph is at most half of the largest density of a cubical graph on the same number of vertices. \\

Let $k$ and $n$ be integers,  $0\leq k\leq n$,  $\cA\subseteq\binom{[n]}{k}$, and $\cB\subseteq\binom{[n]}{k-1}$. 
	Then we define  the graph $Q(n,k;\cA,\cB)$ to be  a bipartite graph with vertex set $\cA\cup\cB$ where $A\in\cA$ is adjacent to $B\in\cB$ if and only if $B\subset A$, i.e., 
	a graph induced by $\cA\cup \cB$ in $Q_n$.
	
	Fix integers $k, i, $ and $j$, where $0\leq k\leq n$, $1\leq i<j\leq n$,  and let $\cA\subseteq\binom{[n]}{k}$ and $\cB\subseteq\binom{[n]}{k-1}$. 
	Let $R_{ij}$ be the shift operator also called {\it compression} operator. That is, for any set $X\in \cA\cup \cB$,
	\begin{align*}
		R_{ij}(X) 	= 	\left\{ 	\begin{array}{ll}
						 		(X-\{j\})\cup\{i\}, 		&\mbox{if $i\not\in X$, ~ $j\in X$, and $(X-\{j\})\cup\{i\}\not\in \cA\cup \cB;$ }\\
								X, 				&\mbox{else.}
							\end{array}\right.
							\end{align*}
	Note that this is a classical shift operator used in proving, for example, the Kruskal-Katona theorem, see  a survey by Frankl and Tokushige  \cite{FT}. For a nice account of the properties of the shift operation, see a summary by Das \cite{Das}. 
	A family  $\cX$ is called {\it compressed}  if for any $i<j$, $R_{ij}(\cX) = \cX$, where $R_{ij}(\cX) = \{ R_{ij}(X): X\in \cX\}$. Note that $|\cX|= |R_{ij}(\cX)|$.

The following lemma shows that the compression doesn't decrease the size of a layered graph.

\begin{lemma}\label{compressed}
Let $k, i, j $ and $n$ be  integers,  $0\leq k\leq n$, $1\leq i<j\leq n$,   $\cA\subseteq\binom{[n]}{k}$, and $\cB\subseteq\binom{[n]}{k-1}$. 
$||Q(n,k;\cA,\cB)|| \leq ||Q(n,k; R_{ij}(\cA), R_{ij}(\cB))||.$
\end{lemma}

\begin{proof}
	Define $ \cA'= {R}_{ij}(\cA)$ and $\cB' ={R}_{ij}(\cB).$
	Let $G= Q(n,k;\cA,\cB)$ and $G' = Q(n,k;\cA',\cB').$  We shall show that $||G|| \leq ||G'||$. 
	Let us denote $R_{ij}(B)$ as $B'$ for any $B\in \cB$ and $R_{ij}(A)$ as $A'$ for any $A\in \cA$. \\
	
	Consider $B\in\cB$ and $i<j$.  If  the set $B-\{j\}\cup \{i\} \in \cB -\{B\}$, we denote this set as $B^*$, i.e.,  $B^*= B-\{j\}\cup \{i\}$ and say that $B^*$ is the {\it successor} of $B$ and $B$ is the {\it predecessor} of $B^*$.  Note that $B^*$ itself doesn't have a successor,  each $B\in \cB$ has at most one successor and at most one predecessor. 
	Let $\cB= \cB_0\cup \cB_1$, where $\cB_0$ consists of all $B$'s  from $\cB$ that have neither successors nor predecessors and $\cB_1= \cB-\cB_0$, a set that can be partitioned into pairs $B, B^*$. We shall treat elements of $\cB_0$ as singletons and split elements of $\cB_1$ into sets of size two consisting of a set and its successor.
	We shall argue that any vertex from $\cB_0$ after the shift has a degree in $G'$ as high as its degree in $G$. In addition, we shall argue that for any pair $\{B, B^*\}$ in $\cB_1$, 
	the number of edges incident to $B$ or $B^*$ after the shift in $G'$ is as large as the number of edges incident to $B$ or $B^*$ in $G$. This will immediately imply that $||G||\leq ||G'||$.
	
	We consider the cases:
	\begin{enumerate}
	\item{} $B\in \cB_0$
	
	\begin{enumerate}
	\item{} $i\not\in B$ and  $j\in B$\\
	In this case $B' = B-\{j\}\cup\{i\}$.
	If $AB\in E(G)$, then $A=B\cup \{t\}$, $t\neq j$. 
	If $t=i$, then $A'=A$ and $A'B' \in E(G')$.
	If $t\neq i$, then $i\not\in A$. If $A' = A-\{j\} \cup \{i\}$, then $A'B'\in E(G')$.  If $t\neq i$ and $A'=A$, we have that $A_t=A-\{j\}\cup \{i\} \in \cA$. 
	Then we have that $A_tB' \in E(G')$.  We see that $\deg_{G'}(B') = \deg_G(B)$.
	
	\item{} $i\in B$ or $j\not\in B$\\
	In this case  $B'=B$.
	If $i\in B$, then for any $A\in \cA$ such that $AB\in E(G)$, $i\in A$, thus $A'=A$. Thus $A'B'\in E(G')$ in this case.
	If $i\not\in B$ and $j\not\in B$  and $AB\in E(G)$, we have two subcases. If $A'=A$, then $A'B'\in E(G')$. Otherwise $j\in A, i\not\in A$.  Then  $A'= A-\{j\}\cup \{i\}$ and $A'B'\in E(G')$.
	\end{enumerate}

	\item{ }  $B\in \cB_1$\\
	In this case we shall consider a pair $B, B^*$ assuming without loss of generality that $B$ has successor $B^*$.
	We shall argue that $\deg_G(B)+\deg_G(B^*) \leq \deg_{G'}(B') +\deg_{G'}({B^*}')$.

	We have $i\not\in B$, $j\in B$, and $B' = B$. Thus, we have that $\{B', {B^*}'\}=\{B, B^*\}$.
	If $AB\in E(G)$ and $AB^*\in E(G)$ then $A=B\cup \{i\}$ and $A'=A$. Then $A'B', A'{B^*}' \in E(G')$.
	If $AB\in E(G)$ and $AB^*\not\in E(G)$, then $j\in A$, $i\not\in A$. Thus either $A'B^* \in E(G')$ or $A'B\in E(G')$ depending whether $A'\neq A$ or $A'=A$, respectively. 
	If $AB\not\in E(G)$ and $AB^*\in E(G)$, then  $j\not\in A$, $i\in A$. Thus $A' = A$ and $A'B'\in E(G')$.
	So, we see that for any $A'\in \cA$, $A'$ sends at least as many edges to $\{B, B^*\}$ in $G'$ as $A$ to   $\{B, B^*\}$ in $G$.
	\end{enumerate} 
	
	This shows that $||G||\leq ||G'||$. Now, we repeat this shift operation for all pairs $i<j$ and produce two compressed families $\cA'' \subseteq \binom{[n]}{k}$ and $\cB''\subseteq \binom{[n]}{k-1}$,  $|\cA|=|\cA''|$, $|\cB|=|\cB''|$, as well as a graph  $G''= Q(n, k, \cA'', \cB'')$ such that $||G''||\geq ||G||$, as desired.
	\end{proof}

\vskip 1cm 

We say that the graph is {\it in the $k$th layer} if its edges are in  the $k$th edge layer of some hypercube.


So, we see that in order to find a largest density of a $t$-vertex layered graph, it is sufficient to find such a density for a compressed graph. 
A special class of compressed set families are those corresponding to the initial interval in colex order. Unfortunately there are compressed families, for example $\{ \{1,2\}, \{1,3\}, \{1,4\}\}$ that do not form an initial interval in colex order. \\

Next, we shall consider only families forming initial segments in colex order.
A set $A$ is less than set $B$ in the colex order if the largest element in the symmetric difference of $A$ and $B$ is in $B$.
	For positive integers $N_A$ and $N_B$, we define the graph $Q(n,k;N_A,N_B)$ to be  the graph $Q(n,k;\cA,\cB)$ where $\cA\subseteq\binom{[n]}{k}$ and $\cB\subseteq\binom{[n]}{k-1}$, are families, of sizes $N_A$ and $N_B$ respectively, that form initial intervals in colex order. 
	We call a graph {\it a colex-interval}  or {\it colex-interval graph} if it is equal to $Q(n,k;N_A,N_B)$ for some $n, k, N_A, $ and $N_B$.

A layered graph in the $k$th layer is  {\it a super-colex-interval} if it is  a colex-interval and equal to $Q(a,k;\cA,\cB)$ for some integer $a$, where  $\binom{[a-1]}{k} \subset \cA  \subseteq   \binom{[a]}{k}$ and 
$\binom{[a-1]}{k-1} \subseteq \cB \subseteq  \binom{[a]}{k-1}$.  In particular,  if $G$ is a super-colex-interval graph on $t$ vertices and  in layer $k$, then $\binom{a-1}{k} + \binom{a-1}{k-1}< t \leq \binom{a}{k} + \binom{a}{k-1}$, i.e., $\binom{a}{k} < t \leq \binom{a+1}{k}$. \\

\begin{lemma}\label{supercompressed}
Let $k$ and $t$ be natural numbers. Let  $G$ be a colex-interval graph in layer $k$ with $|G|=t$. Then the number of edges in $G$ is either at most $2t$ or  at most the number of edges in a super-colex-interval graph on $t$ vertices in layer $k$.  
 \end{lemma}

\begin{proof}
Let $G=Q(n, k; \cA, \cB)$, where  $|G|= |\cA|+|\cB|=t$, for  some $n$, and $G$ has a largest number of edges among colex-interval graphs on $t$ vertices.  We can assume that $3\leq k \leq n-2$, because otherwise the degrees of vertices in one part of $G$ are at most $2$, so $||G|| \leq 2t$ and we are done.\\

Since $G$ is a colex-interval,  $\cA$  and $\cB$ are initial segments in colex order.  We assume also that $\cA$ and $\cB$ are non-empty.
 Thus,
  $$\binom{[a-1]}{k} \subset \cA \subseteq \binom{[a]}{k} \mbox{ and }  \binom{[b-1]}{k-1} \subset \cB \subseteq  \binom{[b]}{k-1},$$ for some positive integers $a$ and $b$.
If $b=a$ or $\cB=\binom{[a-1]}{k-1}$, then $G$ is a super-colex-interval.  Otherwise we shall find a contradiction.  We shall be treating $\cA$ and $\cB$ as linearly ordered sets with respect to colex order. \\

Assume that  $b>a$. Then any vertex $B\in\cB$ that contains $b$ has no neighbors in $\cA$. We can  replace $B$ with $A'$, the member of $\binom{[n]}{k}-\mathcal{A}$ which is smallest in colex order. Then $A'$ has some neighbors in $\cB-\{B\}$, contradicting the maximality of $||G||$.\\

Now assume that $b\le a-1$ and $\cB\neq \binom{[a-1]}{k-1}$.
In this case we can take $n=a$. We assumed in the beginning of the proof that $k\leq n-2 = a-2$.
Let $A$ be the last vertex of $\cA$ in colex order. Note that $a\in A$, thus $A$ has at most one neighbor in $\cB$.
We replace $A$ with the vertex 
$B' \in \binom{[a-1]}{k-1}$ such that $B'\not\in \cB$ and it follows the last member of $\cB$ in colex order.
Since $\cA$ contains all $k$-element subsets of $[a-1]$ and $a\geq k+2$, we see that $B'$ has at least two neighbors in $\cA-\{A\}$.
This results in a graph on a larger number of edges than $G$ and that is a colex-interval, a contradiction.
\end{proof}

\vskip 1cm

\begin{proposition}\label{super}
If $G$ is a layered graph on $t$ vertices that is a colex-interval graph, then $||G|| \leq \frac{1}{4} t \log t (1+o(1))$. 
\end{proposition}

\begin{proof}

Let  $G$ be in layer  $k$, for some $k$.
We can assume by Lemma \ref{supercompressed} that $G$ is a super-colex-interval.

Let $x$ be the real number such that $t = \binom{x}{k}$.  Then $G \subseteq Q_{a}$ for  $a$ satisfying $a < x \leq a+1$.  Since $G$ is in layer $k$ of $Q_a$, we have that $k\leq a$. \\

{\it Case 1. ~}  $2k-2  \leq x \leq 2k+2 $\\
In this case $ 2k-3 \leq a \leq 2k+2 $.
Then $t = \binom{2k}{k}C(1+o(1))$, where $\frac{1}{4} \leq C \leq 4$.  In particular,    $k= \frac{1}{2}\log t (1+o(1))$.  
The degree of any vertex of $G$ from layer $k$ is at most $k$, the degree of any vertex of $G$ in layer $k-1$ is at most $a-k+1 \leq  (2k+2)- k+1 = k+3$.   So $||G|| \leq (k+3) t/2 = \frac{1}{4}t\log t (1+o(1))$, as desired.\\

{\it Case 2. ~}  $x> 2k+2$\\
  In particular, $t = \binom{x}{k} > \binom{2k}{k}$.
Let $k'$ be the integer such that $\binom{2(k'-1)}{(k'-1)}<t\leq \binom{2k'}{k'}$. In particular $k'> k$.
Let $G''$ be obtained by shifting $G$ to layer $k'$, i.e., 
$V(G'') = \big\{v \cup \{a+1, \ldots, a+ (k'-k)\}:  v\in V(G)\big\}$. We have that $||G''||=||G||$ and $|G''|=|G|=t$.
Lemma \ref{supercompressed} gives a graph $G'$ that is super-colex-interval in layer $k'$,  and such that  $|G'|=t$ and $||G'||\geq ||G''||$.  Let $x'$ be the real number such that $t = \binom{x'}{k'}$. By the choice of $k'$  we have
$\binom{2k'-2}{k'-1} < \binom{x'}{k'} \leq \binom{2k'}{k'}$. 
The second inequality implies that  $x' \leq 2k'$.
We shall use the first inequality to show that $x' \geq 2k'-2$. If not, then  $x' <2k'-2$ and  $t = \binom{x'}{k'} < \binom{2k'-2}{k'} < \binom{2k'-2}{k'-1}$, a contradiction.
So, $2k'-2 \leq x' \leq 2k'$ and we are done by Case 1 with $k$ and $x$ replaced by $k'$ and $x'$. \\

{\it Case 3.~ }  $x< 2k-2$\\
Recall that $a<x$ and $k\leq a$, so in particular $k\leq a\leq  2k-3$ in this case. Then consider a vertex-wise complement  $G''$ of  $G$, i.e., an induced subgraph of $Q_a$ with a vertex set $\{[a] -v: v\in V(G)\}$. Then $G''$ is in the layer $k''= a+1-k$, it is isomorphic to $G$, so  $|G_1|=t$ and  $||G_1|| = e(t)$. 
Let $y$ be the real number such that $t = \binom{y}{k''}$.   Assume as before that $G''$ is a super-colex-interval.
If $y\ge 2k'' -2$, we are done by Cases 1 and 2.
So assume that $y<2k'' -2= 2a-2k$. Let $b=2k-a$. Note that $3\leq b\leq k $, 
$2a-2k=a-b$ and $k''=k-b+1$. Then $$\binom{a}{k}<\binom{x}{k}= t = \binom{y}{k''}<\binom{2a-2k}{k''}=\binom{a-b}{k-b+1}.$$
We have for any integers $0<t\leq s$ that $\binom{s+1}{t+1} >\binom{s}{t}$. 
Thus  $$\binom{a}{k}>\binom{a-b+1}{k-b+1}>\binom{a-b}{k-b+1},$$ a contradiction.\\

Therefore  $||G|| \leq \frac{1}{4} t \log t (1+o(1))$.\\

Note that if $G'$  is a middle edge layer of a hypercube $Q_n$ for some even $n$,  then $|G'| = t =\binom{n}{n/2} + \binom{n}{n/2-1}$ and  $||G'|| = \binom{n}{n/2} \frac{n}{2} = \frac{1}{4} t \log t (1+o(1))$. This implies that the largest size of a $t$-vertex layered graph is  $\frac{1}{4} t \log t (1+o(1))$, for any $t $ expressible as the sum $\binom{n}{n/2} + \binom{n}{n/2-1}$, for some even $n$. This shows that the upper bound in Proposition \ref{super} is tight for infinitely many values of $t$.
\end{proof}


\section{Conclusions}\label{conclusions}

The focus of this paper is to investigate the class of layered graphs and their Tur\'an density in the hypercube. Recall that graphs that are not layered have positive Tur\'an density in a hypercube.
First, we developed a characterisation of layered graphs in terms of very nice colorings, that is a convenient tool to analyse them.
Then, we proved that any odd subdivision of a complete graph is layered and has  zero Tur\'an density. Similarly, we showed that any even subdivision of any complete bipartite graph is layered, and for such a $k$-subdivision, where $k$ is divisible by $4$ and $k\geq 8$,  it also has  zero Tur\'an density.  This leaves first question:\\

{\bf Question 1.~} Which graphs  out of $T_2(K_{t,t})$, $T_4(K_{t,t})$, and $T_6(K_{t,t})$ have zero Tur\'an density for any $t$?\\

In addition, we showed that there are some cubical graphs that have girth $8$ and that are not layered. In particular, there are graphs of girth $8$ and of positive Tur\'an density in the hypercube.
This extends known results on graphs of girth $6$ and leads to another question:\\

{\bf Question 2.~} Are there graphs of arbitrarily large girth that are cubical but not layered?

As mentioned in the introduction, very recently this question was answered in the positive by Behague, Leader,  Morrison,  and Williams \cite{BLMW}.\\

Since the density of layered graphs could be close to the density of general cubical graphs, it  seems to be difficult  to find such a graph using direct  probabilistic methods. Nevertheless, the following question is of independent interest:\\

{\bf Question 3.~} What is the largest number of edges in a layered graph on $t$ vertices  for any positive integer $t$? \\

Graham \cite{G}, see also Bollob\'as \cite{Bo}, Hart \cite{H1}, and Chung, F\"uredi, Graham, and Seymour  \cite{CFGS},  determined the largest possible size of a cubical $t$-vertex graph by considering edge-cuts that are matching corresponding to color classes of nice colorings.   
Using Theorem \ref{very-nice} we have that any color class in a very-nice coloring of a layered graph is a cut that is an induced matching. This property might allow one to determine the largest density of a layered graph exactly. 
Although we did not manage to find the largest number of edges in a $t$ vertex layered graph even asymptotically, we believe that the answer should be $\frac{1}{4} t \log t (1+o(1))$, i.e., 
half of the corresponding quantity in case of cubical graphs. 
This question is related  to a class of classical isoperimetric questions since maximising the number of edges in an induced subgraph of a regular graph is equivalent to minimising the number of edges ``leaving'' this subgraph. Finally, we remark that it was proved by Haussler et al. \cite{Ha1, Ha2}, that the largest number of edges in a subgraph of a hypercube induced by $t$ vertices is at most $t$ times the VC-dimension of the set family corresponding to the vertex set.  \\

We made  modest progress towards determining the extremal number of $C_{10}$ in $Q_n$, the remaining case for cycles in a hypercube for which it is not known whether the Tur\'an density is zero or not. 
We proved that $C_{10}$ definitely behaves  differently  from known cycles of zero Tur\'an density in its extremal function, i.e., $\ex(Q_n, C_{10}) = \Omega (n2^n/ \log^b n)$, $b>0$, whereas for any other cycle $C$ of zero Tur\'an density
$\ex(Q_n, C) = O(n^a2^n)$, for some $a<1$.  After this paper was accepted for publication,  Grebennikov and Marciano \cite{GM} proved  that $C_{10}$ has positive Tur\'an density on the hypercube.\\

We note that the bounds on extremal numbers for subdivisions we obtain could be improved using a more efficient embedding. In Appendix A, we recall a general approach introduced by  Chung that might give better upper bounds for some $1$-subdivisions.
Finally, by explicitly constructing partite embeddings of subdivisions, we came up with a quite symmetric way to embed vertices of a hypercube in a layer of a larger hypercube such that adjacent vertices are embedded into pairs of vertices at a  fixed distance.  As it might be of independent interest, we present this construction in  Appendix B. \\

\section{Acknowledgements} The research of the first and the third authors was supported in part by a DFG grant FKZ AX 93/2-1. First author thanks Iowa State University for hospitality during research visit.  The research of the second author was partially supported by a grant from the Simons Foundation (\#709641). The authors thank Chris Cox, Emily Heath, and Bernard Lidick\'y for helpful  discussions, and the anonymous referee for constructive comments on the manuscript. In addition, the authors thank Imre Leader for comments and authors Alexandr Grebennikov and Jo\~ao Pedro  Marciano for pointing out that the constant in Lemma \ref{lem:C6-} given by the proof is $4$ and not $3$.

\section{Conflict of interest} The authors have no relevant financial or non-financial interests to disclose.

\section{Appendix A:~  Another upper bound on extremal number for subdivisions}

\begin{theorem}
Let  $H'$ be a bipartite graph such that  $H=T_{1}(H')$ is cubical. Then $\ex(Q_n, H)= o(||Q_n||)$. 
\end{theorem}

\begin{proof}
We shall use a typical argument introduced by Chung \cite{chung}.  Fix any positive constant $c$ and consider a spanning subgraph $G$  of $Q_n$ with $c||Q_n||$ edges.  Then the average degree of $G$ is $cn$.
For each vertex $v$, create an auxiliary graph $G_v$ with vertex set $N(v)$ in $Q_n$ and two vertices $x, x' \in N(v)$ adjacent in $G_v$ if and only if  there is a vertex $w\neq v$ and 
two edges $wx, wx'$ in $G$.  Note that $w\not\in N(v)$ because $Q_n$ is triangle-free. We claim that there is a vertex $v$ such that $||G_v||\geq c'n^2$ for a positive constant $c'$.
\\

Note that there is no copy of $K_{2,3}$ in $Q_n$.  Moreover, for any two vertices at distance $2$ in $Q_n$ there are exactly two paths of length $2$ in $Q_n$ having these two vertices as endpoints. Thus, for each path $xwy$ of length $2$ in $G$  there is a unique vertex $v$ such that  $xy\in E(G_v)$. In addition, for any edge $xy\in E(G_v)$ there is exactly one path $xwy$, $w\neq v$ in $G$. So, the set of  edges of  all $G_v$'s, $v\in Q_n$  is in a bijective correspondence with the set of paths of length $2$ in $G$.  The number of such paths is $\sum_{u\in V(Q_n)} \binom{d(u)}{2} \geq  \binom{cn}{2} 2^n$. Thus, there is a $v$ such that $||G_v|| \geq \binom{cn}{2} 2^n/2^n  = c' n^2$, for a positive constant $c'$.
Since $H'$ is bipartite,  $\ex(n, H') =o(n^2)$, thus $G_v$ contains $H'$ as a subgraph. 
For any two distinct edges $e$ and $e'$ of this copy of $H'$, there are vertices $w, w' \not\in \{v\}\cup N(v)$ such that $w$ and the endpoints of $e$ form a path of length $2$. Similarly $w'$ and the endpoints of $e'$ form a path length $2$ with $w$ and $w'$ being central vertices on these paths.  Note that $w$ and $w'$ are distinct since $K_{2,3}$ is not a subgraph of $Q_n$.   Thus $G$ contains $T_1(H')$ as a subgraph. This implies that $\ex(Q_n, H)= o(||Q_n||)$. 
\end{proof}

\section{Appendix B: ~ Embedding of  vertices of $Q_n$ into two consecutive layers of $Q_N$  with adjacent vertices in $Q_n$ at a fixed given  distance in $Q_N$.}

While we presented a layered embedding of the subdivision of any bipartite graph in the main body of the paper, 
here we present a more symmetric embedding of $V(Q_n)$. It in turn could be extended to embed subdivisions of $Q_n$ and not only their branch vertices.
This contributes to a large body of research on embeddings in hypercubes that focuses on more efficient embeddings, see for example \cite{LS, FHT, AHL, VS, BCHRS, AS, MARR,  H, B, HS, X, Chen, LT}.

\begin{theorem}
For any integer $m\geq 2$ and any positive integer $n$, there exist an integer $N$ and  a function $F: V(Q_n) \rightarrow V(Q_N)$, such that 
for any two vertices $u$ and $v$ which are adjacent in $Q_n$,  $d_H(F(u), F(v))=m$ and $F$ maps all vertices of $Q_n$ either in one vertex layer of $Q_N$ (if $m$ is even) or 
in two consecutive vertex layers of $Q_N$ (if $m$ is odd).
\end{theorem}
\begin{proof}

Here, we shall present functions $f, f', f_k$ mapping $V(Q_n)$ into the vertex set of some larger hypercube, for a fixed $k\in\mathbb{N}\cup \{0\}$  such that 
for any two adjacent in $Q_n$ vertices $u$ and $v$,  $d_H(f(u), f(v))= d_H(f'(u), f'(v))=3$ and $d_H(f_k(u), f_k(v))=2k+2$. 
Moreover, both $f$ and $f'$ map vertices of $Q_n$ into two consecutive vertex layers  and $f_k$ maps vertices of $Q_n$ into one layer. We shall then define $F$ based on one of the functions $f, f'$, or $f_k$.\\

For any vector $w$, let  $w[i]$ denote the $i$th component of $w$ and $||w||$ denote the number of $1$'s in $w$. \\

Let $[(2k+2)n]$ be split into $n$ consecutive intervals of length $2k+2$.   For a binary vector $w$ of length $(2k+2)n$, let $w[[i]]$ be $w$ restricted to the $i$th interval of length $2k+2$. Formally, $w[[i]] = w[(2k+2)(i-1)+1] w[(2k+2)(i-1)+2] \cdots w[(2k+2)i]$. We define $f_k: V(Q_n) \rightarrow V(Q_{(2k+2)n})$ as follows:

\begin{equation} \nonumber
f_k(v)[[i]] = \begin{cases}
0101\cdots 01,  \mbox{ if }  v[i]=0, \\
1010\cdots 10,  \mbox{ if }  v[i]=1.
\end{cases}
\end{equation}

\vskip 1cm

Let $[2n+1]$ be split into $n$ consecutive intervals of length $2$ and one last element. For a binary vector $w$ of length $2n+1$, let $w[[i]]$ be a triple corresponding to $w$ restricted to the $i$th interval and the last element, i.e., $w[[i]]= w[2i-1] w[2i] w[N']$.   We define $f: V(Q_n)\rightarrow V(Q_{2n+1})$ as follows:

\begin{equation}\nonumber
f(v)[[i]] = \begin{cases}
010,  \mbox{ if }  v[i]=0  \mbox{ and  }  ||v|| \mbox{ is even, }  \\
100,  \mbox{ if }  v[i]=1  \mbox{ and  }  ||v|| \mbox{ is even, }  \\
011,  \mbox{ if }  v[i]=0  \mbox{ and  }  ||v|| \mbox{ is odd, }  \\
101,  \mbox{ if }  v[i]=1  \mbox{ and  }  ||v|| \mbox{ is odd. }  \\
\end{cases}
\end{equation}

It is clear here that if $u$ and $v$ are adjacent in $Q_n$, the images $f(u)$ and $f(v)$ are at Hamming distance $3$.

\vskip 1cm	

Let $[3n]$ be split into $n$ consecutive intervals of length $3$.  For a binary vector $w$ of length $3n$, let $w[[i]]$ be a triple corresponding to $w$ restricted to the $i$th interval  $w[[i]]= w[3i-2] w[3i-1] w[3i]$.  We define $f':V(Q_n)\rightarrow V(Q_{3n})$ as follows:

Let 
\begin{equation} \nonumber
f'(v)[[i]] = \begin{cases}
010,  \mbox{ if }  v[i]=0, \\
100,  \mbox{ if }  v[i]=1 ~ \mbox{ and   } ||v|| \mbox{ is even}, \\
101,  \mbox{ if }  v[i]=1,  v[j]=0 \mbox{ for any } j>i, ~ \mbox{ and   } ||v|| \mbox{ is odd},\\
100,  \mbox{ if }  v[i]=1,  v[j]=1 \mbox{ for some } j>i, ~ \mbox{ and   } ||v|| \mbox{ is odd}.
\end{cases}
\end{equation}		
		
Assume that $v$ and $u$ are adjacent in $Q_n$ and differ in position $i$ such that $v$ is zero in this position.
We shall verify that the distance between $f(v)$ and $f(u)$ is $3$.
Note that $f(v)$ and $f(u)$ coincide in all triples corresponding to $0$'s of $u$. Moreover, they coincide on those triples $\ell$, where $u[\ell]=1$, $j\neq i$, and $\ell$ is not a position of the last $1$ of $u$ or $v$.  Let $j$ be the last position of $1$ in $u$. Note that $i$ could be equal to $j$.  \\

If $w(v)$ is even, then $w(u)$ is odd and $f(u)[[j]]=101$. If $i=j$, then $f(v)[[j]]=010$ and $f(u)$ and $f(v)$ coincide in all other triples.  If $i<j$,  then $u[j]=v[j] =1$, $f(u)[[j]]=101$, $f(v)[[j]]=100$, 
$f(u)[[i]]=100$, and $f(v)[[i]]=010$. On all other triples $f(v)$ and $f(u)$ coincide. We see that in both cases $f(u)$ and $f(v)$ are at distance $3$.\\

If $w(v)$ is odd, then $w(u)$ is even and $f(u)[[i]]=100$. 
Let $k$ be the last position of $1$ in $v$. So, $f(v)[[k]]= 101$. 
We also have $f(u)[[k]]=100$ and $f(v)[[i]]= 010$. Then $f(u)$ and $f(v)$ are at distance $3$.\\~\\

Now, let $m\geq 2$ be given. If $m$ is even, let $m=2k+2$, for non-negative integer $k$. Then let $F(u) = f_k(u)$ for any $u\in V(Q_n)$. 
If $m$ is odd and $m=3$ let $F(u) = f(u)$ for any $u\in V(Q_n)$. 
If $m$ is odd and $m= 3 +2\ell$ for some positive integer $\ell$, we define $F$ by considering either $f$ or $f'$ and adjusting $2\ell$ coordinates to each embedded vertex 
that are $0\cdots 0 1 \cdots 1$ or $ 1\cdots 1 0\cdots 0$, depending whether the vertex is embedded in one layer of the other. 
Formally, in case of $f$, for example, let  $N= 2n+1+ 2\ell$ and let for any vertex $u$ of $Q_n$, $F(u)$ restricted to the first $2n+1$ coordinates be $f(u)$. 
In addition, if $w(u)$ is even, let the last $2\ell$ coordinates of $F(u)$ be $0\cdots 0 1 \cdots 1$ and 
if $w(u)$ is odd, let the last $2\ell$ coordinates of $F(u)$ be $1\cdots 1 0 \cdots 0$,  with $\ell$ $0$'s and $\ell$ $1$'s respectively. 
\end{proof}

\end{document}